\documentclass[10pt]{article}
\usepackage{amsmath}
\usepackage{amsfonts}

\begin{document}

\numberwithin{equation}{section}
\newtheorem{theorem}{Theorem}[section]
\newtheorem{proposition}{Proposition}[section]
\newtheorem{lemma}{Lemma}[section]

\newenvironment{proof}[1][Proof.]{\begin{trivlist}
\item[\hskip \labelsep {\bfseries #1}]}{\end{trivlist}}

\newenvironment{definition}[1][Definition.]{\begin{trivlist}
\item[\hskip \labelsep {\bfseries #1}]}{\end{trivlist}}

\newenvironment{example}[1][Example.]{\begin{trivlist}
\item[\hskip \labelsep {\bfseries #1}]}{\end{trivlist}}

\newenvironment{remark}[1][Remark.]{\begin{trivlist}
\item[\hskip \labelsep {\bfseries #1}]}{\end{trivlist}}

\newcommand{\qed}{\nobreak \ifvmode \relax \else
      \ifdim\lastskip<1.5em \hskip-\lastskip
      \hskip1.5em plus0em minus0.5em \fi \nobreak
      \vrule height0.75em width0.5em depth0.25em\fi}

\newcommand{\be}{\begin{equation}}
\newcommand{\ee}{\end{equation}}
\newcommand{\bal}{\begin{aligned}}
\newcommand{\eal}{\end{aligned}}
\newcommand{\p}{\partial}
\newcommand{\f}{\frac}
\newcommand{\m}{\mathcal}
\newcommand{\wt}{\widetilde}
\newcommand{\wh}{\widehat}
\newcommand{\n}{\nabla}
\newcommand{\mb}{\mathbb}

\title{Nonlinear PDEs and Scale Dependence}

\author{Garry Pantelis}

\date{}

\maketitle

\begin{abstract}
The properties of nonlinear PDEs that generate filtered solutions are
explored with particular attention given to the constraints on the
residual term. The analysis is carried out for nonlinear PDEs with
an emphasis on evolution problems recast on space-time-scale. We
examine the role of approximation that allow for the generation of
solutions on isolated scale slices.
\end{abstract}

\section{Introduction}

We are interested in applications, particularly of the evolution type, that are governed by nonlinear PDEs that generate solutions that exhibit fluctuations at all spatial/temporal scales.
It is often the case that one is interested in obtaining a macroscopic solution that is defined as some kind of filtered representation of the dependent variables associated with some scale of resolution.
Consequently, the macroscopic equations satisfied by the filtered variables will depend not only on space-time but also on a parameter associated with scale.
We start by example and then generalise in the following sections.

Let $x=(x^1, ... , x^n)$ be the cartesian coordinates of $\mb{R}^n$ ($n=1,2,3$), representing the space domain, and denote the time coordinate by $t \in \mb{R}_+$ ($\mb{R}_+ =$ the set of positive real numbers excluding $0$).
In order to define the filtered variables we introduce the scale parameter $\eta \in I$ on the scale interval $I=(0,\eta_0)$ where $0 < \eta_0 << 1$.
We work on space-time-scale $M = \mb{R}^n \times \mb{R}_+ \times I \subset \mb{R}^m$, $m={n+2}$, and introduce the linear operator
\be
\m{L} = \f{\p~}{\p \eta} - \triangle
\label{filterop}
\ee
where $\triangle = \n \cdot \n$ and $\n$ is the spatial gradient operator.
A filtered scalar variable $\phi :M \rightarrow \mb{R}$ can now be defined by
\be
\m{L} \phi = 0 \qquad (x,t,\eta) \in M \,.
\label{filter}
\ee
Suppose that there exists $\Phi(x,t)$ ($\| \Phi \|_{L_2(\mb{R}^n)} < \infty$) such that $\lim_{\eta \rightarrow 0^+} \phi = \Phi(x,t)$.
If we can also set $\lim_{|x| \rightarrow \infty} \n \phi = 0$ then $\phi$ is uniquely defined by
\be
\phi (x,t,\eta) = \int_{\mb{R}^n} G(x-\hat{x},\eta) \Phi (\hat{x},t) \, d \hat{x}^1 \dots d \hat{x}^n
\label{i2}
\ee
where $G(x,\eta) = (4 \pi \eta)^{-n/2} \exp [- |x|^2 / (4 \eta )]$ with the property $\int_{\mb{R}^n} G(x,\eta) \, d x^1 \dots d x^n = 1$.
Hence $\phi(x,t,\eta)$ becomes the Gaussian filter of $\Phi(x,t)$.
In this way we can identify $\eta = \beta \delta^2$, where $\delta$ is a parameter associated with the characteristic spatial resolution and $\beta$ is a constant associated with the filter width.

A major sticking point with this approach is the need to assume the existence of a regular limiting solution, as $\eta \rightarrow 0^+$, associated with the space-time pointwise representation of the dependent variables. 
Here we shall depart from the approach made in \cite{Pan04} and \cite{Pan05} by attempting to work entirely in $M$ and avoid any reference to a limiting fully resolved solution on space-time.

We denote by $\m{R}(M,\mb{R}^N)$, $N \geq 1$, the set of smooth vector functions on $M$, i.e. smooth mappings $M \rightarrow \mb{R}^N$.
For the case $N=1$ members of $\m{R}(M,\mb{R}^N)$ are smooth scalar functions.

Consider now an ideal fluid whose variables are defined on $M$ as follows. 
Let the fluid velocity $v \in \m{R}(M,\mb{R}^n)$ be a filtered field, defined by
\be
\m{L} v = 0 \qquad (x,t,\eta) \in M \,,
\label{ff}
\ee
and satisfy the continuity condition
\be
\n \cdot v = 0 \qquad (x,t,\eta) \in M \,.
\label{df}
\ee
We now construct the macroscopic equations from (\ref{ff}) and (\ref{df}).

On each scale slice $M|_{\eta=\text{const}}$, in a time interval $\m{I} \subset \mb{R}$ a fluid element follows a path under the flow characterized by a mapping $\gamma : \m{I} \rightarrow M|_{\eta=\text{const}}$ (integral curve of $v$) so that
\be
v \circ \gamma (t) = \f{d x \circ \gamma (t)}{dt} \,.
\ee
The acceleration, $a$, of the same fluid element can be defined along the flow path $\gamma(t)$ by
\be
a \circ \gamma(t) = \f{d^2 x \circ \gamma (t)}{dt^2} = \f{d v \circ \gamma (t)}{dt} \,.
\ee
Alternatively we may write $a \in \m{R}(M,\mb{R}^n)$ as
\be
a = \f{\p v}{\p t} + v \cdot \n v \,.
\ee
We introduce the symmetric second order tensor $\sigma$ whose components are given by
\be
\sigma^{ab} = \sum_{c=1}^n \f{\p v^a}{\p x^c} \f{\p v^b}{\p x^c} \, \quad 1 \leq a,b \leq n \,.
\label{sig}
\ee
We also define
\be
s= - 2 \, \n \cdot \sigma \,.
\label{srce}
\ee
The properties assigned to $v$ above dictate that the acceleration must satisfy the following constraint.

\begin{proposition}
If  $v \in \m{R}(M,\mb{R}^n)$ satisfies (\ref{ff}) and (\ref{df}) then
\be
\m{L} a = s \qquad (x,t,\eta) \in M \,,
\label{ffa}
\ee
where $s$ is given by (\ref{srce}).
\end{proposition}

\begin{proof}
We have
\be
\m{L} a = \m{L} (\f{\p v}{\p t} + v \cdot \n v) = \f{\p ~}{\p t} (\m{L} v) + \m{L} (v \cdot \n v) \,.
\label{p111}
\ee
By a straightforward calculation
\be
\m{L} (v \cdot \n v) = (\m{L} v) \cdot \n v + v \cdot \n (\m{L} v) + s
\label{p112}
\ee
where we make use of the fact that $v$ is divergent free to express the source term $s$ in the divergent form (\ref{srce}).
The required result follows from (\ref{p111}) and (\ref{p112}) and the fact that $v$ is a filtered field
.\qed
\end{proof}

We can always introduce a filtered scalar function $p \in \m{R}(M,\mb{R})$,
\be
\m{L} p = 0 \qquad (x,t,\eta) \in M \,,
\ee
and write
\be
a = - \n p + r \,,
\label{defna}
\ee
for some $r \in \m{R}(M,\mb{R}^n)$.
Noting that $\m{L} a = - \m{L} \n p + \m{L} r = - \n \m{L} p + \m{L} r = \m{L} r$, we have constructed from (\ref{ff}) and (\ref{df}) the system of macroscopic equations:
\be
\n \cdot v = 0 \qquad (x,t,\eta) \in M \,,
\label{mac01}
\ee
\be
\f{\p v}{\p t} + v \cdot \n v + \n p = r \qquad (x,t,\eta) \in M \,,
\label{mac02}
\ee
\be
\m{L} r = s \qquad (x,t,\eta) \in M \,,
\label{mac03}
\ee
where $s$ is given by (\ref{srce}).

We can think of the residual $r$ as incorporating the effects of the large scale viscous and stress/strains of the fluid that are manifestations of the filtering process.
In the case that we can assume the existence of a limiting pointwise representation of the filtered variables we can impose, under some appropriate norm $\| \cdot \|$, the condition that $\lim_{\eta \rightarrow 0^+} \|r - \nu \triangle v\| = 0$, for some constant $\nu$.
In this way we recover, at least in the generalised sense, the Navier-Stokes equation from the macroscopic equations (\ref{mac01})-(\ref{mac02}) as a limiting formulation.
Since the notion of viscosity is itself meaningful only in some locally averaged sense, under the continuum fluid model the inviscid limit, $\lim_{\eta \rightarrow 0^+} \|r\| = 0$, is more natural.

\section{Macroscopic Equations}

We now attempt to generalise avoiding application specific issues.
The system of nonlinear PDEs considered will be of first order but it should be apparent that the analysis can incorporate higher order systems.
As in the previous section, let $x=(x^1, ... , x^n)$ be the cartesian coordinates of $\mb{R}^n$ ($n=1,2,3$), representing the space domain, and denote the time coordinate by $t \in \mb{R}_+$ ($\mb{R}_+ =$ the set of positive real numbers excluding $0$).
Let $M = \mb{R}^n \times \mb{R}_+ \times I \subset \mb{R}^m$, $m={n+2}$, where $I=(0,\eta_0)$ is the scale interval, $\eta \in I$ is the scale parameter and $0 < \eta_0 << 1$.
We shall write $\xi=(x,t,\eta)=(\xi^1, \dots , \xi^m)$, where we have set $\xi^a = x^a$ ($1 \leq a \leq n$), $\xi^{n+1}=t$ and $\xi^m=\xi^{n+2}=\eta$ is the scale parameter.

For any $D \subseteq \mb{R}^{N_1}$, $N_1 > 0$, we denote by $\m{R}(D,\mb{R}^{N_2})$, $N_2 >0$, the set of smooth ($C^\infty$) maps $u: D \rightarrow \mb{R}^{N_2}$.
For $u=(u^\alpha)=(u^1,\dots,u^N) \in \m{R}(M,\mb{R}^N)$ we shall denote $u_i=\p_i u =(u_i^\alpha)$, $u_{ij}=\p_i \p_j u = (u_{ij}^\alpha)$, $\dots$, with the shorthand notation $\p_i = \p \,/\p \xi^i$.
Throughout, unless otherwise stated, we use the range $1 \leq i,j,k \leq m$ for lower case Latin indices and the range $1 \leq \alpha, \beta ,\gamma \leq N$ for the lower case Greek indices.
For easier identification of certain important operators we use the indices $t$ instead of $n+1$ and $\eta$ instead of $m=n+2$, e.g. we write $\p_t$ instead of $\p_{n+1}$ and $\p_\eta$ instead of $\p_{n+2} = \p_m$.

We shall express our macroscopic equations in terms of a core function of $u$ and $u_i$ that is form invariant with respect to scale and a residual term.
Let $F \in \mb{R}^N$ such that $F(u,u_i)$ is a smooth vector function with respect to its indicated arguments.
Given any $u \in \m{R}(M,\mb{R}^N)$ there exists an $r \in \m{R}(M,\mb{R}^N)$ such that
\be
F (u,u_i) = r \qquad \xi \in M \,.
\label{me2}
\ee
Alternatively, we may regard (\ref{me2}) as a system of $N$ first order PDEs for the unknown dependent vector function $u(\xi)$ if $r(\xi)$ is known.
We shall henceforth refer to $r$ as the residual and $F(u,u_i)$ as the core function.
The core function may be regarded as retaining the structure of the $N$ PDEs that are associated with the fully resolved system as shall be outlined below.
The construction of the core function in a wider context is application specific and follows along the lines discussed in the previous section.
We shall return to this application in a later section.
Our first task is to find constraints on $r(\xi)$ such that the solutions of (\ref{me2}) have well defined filter properties.

We introduce the vector field operators $V_i$ defined by
\be
V_i = \p_i + u_i^\alpha \dot{\p}_\alpha + u_{ij}^\alpha \dot{\p}_\alpha^j + u_{ijk}^\alpha \dot{\p}_\alpha^{jk} + \dots
\label{me4}
\ee
with the notation
\be
\p_i = \f{\p}{\p \xi^i} , \quad
\dot{\p}_\alpha = \f{\p~}{\p u^\alpha} \,, \quad
\dot{\p}_\alpha^i = \f{\p~}{\p u_i^\alpha} \,, \quad
\dot{\p}_\alpha^{ij} = \f{\p~}{\p u_{ij}^\alpha} \,, \quad \dots
\label{me3}
\ee
Here and throughout we use the summation convention of repeated upper and lower indices over the assigned ranges for the given indice type given above.
For $u \in \m{R}(M,\mb{R}^N)$ the operator $V_i$ acting on any $F \in \mb{R}^N$ using the more general representation $F(\xi,u,u_i, u_{ij}, \dots )$ generalises the partial derivative operator $\p_i$ acting on members of $\m{R}(M,\mb{R}^N)$.
($V_i F$ is sometimes referred to as the total derivative of $F$ with respect to $\xi^i$). 
The action of $V_i$ on any $w \in \m{R}(M,\mb{R}^N)$, $w=w(\xi)$, gives the partial derivative of $w$ with respect $\xi^i$, i.e. $V_i w = \p_i w$.
The vector field operators $V_i$ may be truncated up to terms associated with the largest order of the partial derivatives of $u$ present in the expressions on which $V_i$ is operating.
In truncated form the vector fields $V_i$ have special geometric significance in the general theory of PDEs \cite{Ede90},\cite{Ede92}.

We introduce the spatial Laplacian operator
\be
\triangle = \n \cdot \n \, ,
\label{me8}
\ee
where $\n = (\p_1 , \cdots , \p_n)$ is the spatial gradient operator. 
The associated operator acting on vector functions $F \in \mb{R}^N$ using the general representation $F(\xi,u,u_i, u_{ij}, \dots )$ is defined by
\be
L = \sum_{b=1}^n V_b V_b \,.
\label{me9}
\ee
The set of filter maps is defined as follows:

\begin{definition} $\m{P}(M,\mb{R}^N)$ is the set of vector functions $u \in \m{R}(M,\mb{R}^N)$ such that
\be
(\p_\eta - \triangle) u = 0 \qquad \xi \in M \,.
\label{me10}
\ee
\end{definition}

We can assume that $\m{P}(M,\mb{R}^N)$ is sufficiently populated that it is of interest to the analysis that follows.
Suppose for the moment that we can assume that there exists $\wt{u} \in \m{R}(\mb{R}^n \times \mb{R}_+,\mb{R}^N)$ satisfying the first order PDEs
\be
F (\wt{u},\p_i \wt{u}) =0 \qquad (x,t) \in \mb{R}^n \times \mb{R}_+
\label{me11}
\ee
for some $F \in \mb{R}^N$ such that $F(u,u_i)$ is smooth with respect to its indicated arguments.
If for some $u \in \m{P}(M,K)$ we have $u |_{\eta=0^+} = \wt{u}$ and $\lim_{|x| \rightarrow \infty} \n u = 0$ then $u (x,t,\eta) = \int_{\mb{R}^n} G(x-\hat{x},\eta) \wt{u} (\hat{x},t) \, d \hat{x}^1 \dots d \hat{x}^n$, where $G(x,\eta) = (4 \pi \eta)^{-n/2} \exp [- |x|^2 / (4 \eta )]$.
Hence $u(\xi)= u (x,t,\eta)$ is the Gaussian spatial filter of $\wt{u} (x,t)$.

The PDEs (\ref{me11}) represent the fully resolved system.
It should be mentioned at this point that in many applications of interest $F$ will not explicitly depend on terms $u_\eta$ (see application below).
Under our assigned range for the lower case Latin indices the representation $F(u,u_i)$ will indicate a possible explicit dependence on $u_\eta$.
The more general representation $F(u,u_i)$ will not effect the calculations below and is used to avoid introducing more notation for the indices and their ranges.

Since for $u \in \m{P}(M,\mb{R}^N)$ we cannot expect that in general $F(u,u_i)$ will vanish identically on $M$  we must introduce a residual term $r \in \m{R}(M,\mb{R}^N)$ as in (\ref{me2}).
The system (\ref{me11}) is recovered from (\ref{me2}) by requiring that $\lim_{\eta \rightarrow 0^+} r = 0$.

In many applications one cannot gaurantee the existence of regular solutions $\wt{u} \in \m{R}(\mb{R}^n \times \mb{R}_+,\mb{R}^N)$ for the fully resolved system and it is at this point that we depart from the assumptions made in \cite{Pan04},\cite{Pan05}.
Our focus is on nonlinear problems that generate solutions that exhibit fluctuations at all space and time scales. 
We propose that such systems can be defined in a meaningful way enirely within space-time-scale avoiding any reference to a fully resolved system.

We shall make extensive use of the vector field operator
\be
W = \p_\eta + (\triangle u^\alpha ) \dot{\p}_\alpha + (\triangle u_i^\alpha) \dot{\p}_\alpha^i + (\triangle u_{ij}^\alpha) \dot{\p}_\alpha^{ij} + \dots
\label{me13}
\ee
We are now in a position to establish the constraints on the residual for filter maps. 

\begin{proposition}
If $u \in \m{P} (M,\mb{R}^N)$ and $r \in \m{R} (M,\mb{R}^N)$ such that $F(u,u_i) - r =0$, $\xi \in M$, where $F \in \mb{R}^N$ and $F(u,u_i)$ is smooth with respect to its indicated arguments, then $(\p_\eta - \triangle) r = s$, $\xi \in M$, where $s = (W - L ) F(u,u_i)$.
\label{thme1}
\end{proposition}

\begin{proof}
For any $u \in \m{R} (M,\mb{R}^N)$ there will always exist an $r \in \m{R} (M,\mb{R}^N)$ such that $F(u,u_i)-r$ vanishes identically on $M$.
We have
\be
(V_\eta - L) (F(u,u_i) - r)=0 \,.
\label{me16}
\ee
Manipulating the left hand side we obtain
\be
\bal
(V_\eta - L) (F(u,u_i) - r) & = (V_\eta - L) F(u,u_i) - (\p_\eta - \triangle) r \\
 & = (V_\eta - W) F(u,u_i) + (W - L) F(u,u_i) - (\p_\eta - \triangle) r \\
 & = (V_\eta - W) F(u,u_i) + s - (\p_\eta - \triangle) r \,,
\eal
\label{me16a}
\ee
where we take $r \in \m{R} (M,\mb{R}^N)$ to mean that it has the representation $r(\xi)$ so that $V_i r = \p_i r$.
It is straightforwad to see that if $u \in \m{P} (M,\mb{R}^N)$ then $V_\eta=W$ and hence the first term in this last expression of (\ref{me16a}) vanishes.
Thus
\be
(V_\eta - L) (F(u,u_i) - r)= -(\p_\eta - \triangle) r + s \,.
\label{me16b}
\ee
The desired result follows from (\ref{me16}) and (\ref{me16b})
.\qed
\end{proof}

We can now write the full system of macroscopic equations for the filtered variables $u \in \m{P}(M,\mb{R}^N)$ as
\be
F(u,u_i) = r \qquad \xi \in M \,,
\label{me18}
\ee
\be
(\p_\eta - \triangle) r = s \qquad \xi \in M \,,
\label{me19}
\ee
where
\be
s = (W-L)F(u,u_i) \,.
\label{me20}
\ee

Given any core function $F \in \mb{R}^N$ such that $F(u,u_i)$ is smooth with respect to its arguments, for any $u \in \m{P}(M,\mb{R}^N)$ there will always exist an $r \in \m{R}(M,\mb{R}^N)$ in (\ref{me18}) satisfying (\ref{me19})-(\ref{me20}).
From the practical viewpoint, we must consider as a separate issue the task of developing algorithms that extract from $\m{P}(M,\mb{R}^N)$ those members that can be regarded as meaningful, in some formal sense, to the specific application under consideration.   

We now establish an important link between the filter error, $\psi$, and the residual equation error, $e$, as defined in the following.

\begin{proposition}
Let $u,\,r \in \m{R} (M,\mb{R}^N)$ such that $F(u,u_i)-r=0$, $\xi \in M$, where $F \in \mb{R}^N$ and $F(u,u_i)$ is smooth with respect to its indicated arguments.
Let $\psi\in \m{R} (M,\mb{R}^N)$ be defined by
\be
\psi (\xi) = (\p_\eta - \triangle)u (\xi) \,.
\label{me29}
\ee
It follows that $\psi$ satisfies the linear system of first order PDEs
\be
C_\beta^{\alpha i} \p_i \psi^\beta +
C_\beta^\alpha \psi^\beta = e^\alpha \qquad \xi \in M \,,
\label{me30}
\ee
where the coefficients are given by $C_\beta^{\alpha i} = \dot{\p}_\beta^{i} F^\alpha (u,u_j)$, $C_\beta^\alpha = \dot{\p}_\beta F^\alpha (u,u_i)$, and
\be
e = (\p_\eta - \triangle )r - s \,,
\label{me30a}
\ee
where $s = (W-L)F(u,u_i)$.
\label{thme2}
\end{proposition}

\begin{proof}
We follow a similar argument to that of the proof of the previous proposition but note that now $r$ need not satisfy (\ref{me19})-(\ref{me20}) and hence $u$ is not necessarily a member of $\m{P}(M,\mb{R}^N)$.
Since $F(u,u_i)-r$ vanishes identically on $M$ we have
\be
(V_\eta - L) (F(u,u_i) - r)=0 \,.
\label{me31}
\ee
The left hand side becomes
\be
\bal
(V_\eta - L) (F(u,u_i) - r) & = (V_\eta - L) F(u,u_i) - (\p_\eta - \triangle) r \\
 & = (V_\eta - W) F(u,u_i) + (W - L) F(u,u_i) - (\p_\eta - \triangle) r \\
 & = (V_\eta - W) F(u,u_i) + s - (\p_\eta - \triangle) r \\
 & = (V_\eta - W) F(u,u_i) - e \,,
\eal
\label{me32}
\ee
where we take $r \in \m{R} (M,\mb{R}^N)$ to mean that it has the representation $r(\xi)$ so that $V_i r = \p_i r$.
The first term in the last expression will not in general vanish because $u$ is not necessarily a member of $\m{P}(M,\mb{R}^N)$.
Using the definitions of the vector field operators $V_\eta$ and $W$, we have the explicit representation
\be
V_\eta - W  = \psi^\alpha \dot{\p}_\alpha + (\p_i \psi^\alpha) \dot{\p}_\alpha^i + (\p_i \p_j \psi^\alpha) \dot{\p}_\alpha^{ij} + \dots
\label{me33}
\ee
The identity (\ref{me30}) then follows from (\ref{me31})-(\ref{me33})
.\qed
\end{proof}

If the residual $r$ satisfies (\ref{me19}), where the source term is given by (\ref{me20}), then $e$ vanishes identically on $M$ and (\ref{me30}) admits the trivial solution from which it follows that $u \in \m{P}(M,\mb{R}^N)$.

The system of PDEs (\ref{me18})-(\ref{me19}) are the exact equations satisfied by the filtered variables $u \in \m{P}(M,\mb{R}^N)$ on space-time-scale $M$.
They are of little use in applications since the presence of the $\p_\eta$ term in the residual equation (\ref{me19}) means that the solution of the system can only be obtained by scale integration from some arbitrary scale slice of $M$ on which $u$ is prescribed.
To overcome this problem we resort to approximation by replacing the residual equation (\ref{me19}) with a constraint on the residual that allows us to generate regular solutions $u$ independently on scale slices $M|_{\eta = \text{const}}$ such that $u$ is a good approximation of some member of $\m{P}(M,\mb{R}^N)$ on that scale slice.

\section{Approximation}

An example of a residual approximation that has been found to be useful in applications is one in which the exact residual equation (\ref{me19}) is replaced by the constraint (see \cite{Pan04},\cite{Pan05})  
\be
\triangle r - \f{r}{\eta} + s =0 \,,
\label{ra}
\ee
where the source term $s$ is given by (\ref{me20}).
The approximate macroscopic system (\ref{me18}) and (\ref{ra}) can be solved independently on any scale slice $M|_{\eta=\text{const}}$ by updating the residual via (\ref{ra}) at each timestep as soon as $u$ is available.
If we can assume that $\lim_{\eta \rightarrow 0^+} | r | =0$ it follows that
\be
|e|= \left |\f{\p r}{\p \eta} - \f{r}{\eta} \right | \leq \f{\eta}{2} \sup_{M} \left | \f{\p^2 r}{\p \eta^2} \right |  \,.
\label{rarre}
\ee
Motivated by a discussion of residual approximations in general, we shall avoid any focus on specific formulations such as (\ref{ra}).

In practical application space-time solutions are generated from the system of PDEs (\ref{me18}) on some scale slice $M|_{\eta = \text{const}}$ while using a residual that is not exact, i.e. one in which $e$ does not vanish identically on that scale slice.
If we enforce (\ref{me18}) with a residual that is not exact then it follows from Proposition \ref{thme1} that $u$ cannot be a member of $\m{P}(M,\mb{R}^N)$.

Our objective here is to investigate how well members of $\m{R}(M,\mb{R}^N)$ serve as approximations to members of $\m{P}(M,\mb{R}^N)$. 
More precisely, suppose that we employ a residual that is based on an approximation of the residual equation (\ref{me19}) for which we know that a regular solution $u \in \m{R}(M,\mb{R}^N)$ for (\ref{me18}) exists.
We are interested in how $u$ will deviate from some member of $\m{P}(M,\mb{R}^N)$ as the scale parameter is increased if the two solutions agree on some arbitrary scale slice.

\begin{proposition}
For some $\varepsilon \in (0,\eta_0)$, let $M_\varepsilon = M|_{\eta \in (\varepsilon,\eta_0)}$.
Let $u \in \m{R} (M_\varepsilon,\mb{R}^N)$ and $\bar{u} \in \m{P}(M,\mb{R}^N)$ such that $\lim_{\eta \rightarrow \varepsilon^+} (u - \bar{u})=0$.
Suppose further that $\psi \in \m{R}(M_\varepsilon,\mb{R}^N)$, defined by $\psi (\xi) = (\p_\eta - \triangle)u(\xi)$, is bounded  on $M_\varepsilon$ and that $\lim_{|x| \rightarrow \infty} \n \bar{u} , \, \n u = 0$.
It follows that
\be
(u - \bar{u})(x,t,\eta) =  \int_\varepsilon^\eta \int_{\mb{R}^n} G(x- \hat{x},\eta - \hat{\eta}) \psi (\hat{x},t, \hat{\eta}) \, d \hat{x}^1 \dots d \hat{x}^n \, d \hat{\eta} \qquad (x,t,\eta) \in M_\varepsilon \,,
\label{a16}
\ee
where $G(x,\eta) = (4 \pi \eta)^{-n/2} \exp [- |x|^2 / (4 \eta )]$.
\label{thme3}
\end{proposition}

\begin{proof}
Setting $\omega = u - \bar{u}$ we have
\be
(\p_\eta - \triangle) \omega = \psi \qquad \xi \in M_\varepsilon \,.
\label{a19}
\ee
We also have 
\be
\omega|_{\eta = \varepsilon^+} = 0 ~; \qquad \lim_{|x| \rightarrow \infty} \n \omega = 0 \,.
\label{a20}
\ee
The solution of (\ref{a19})-(\ref{a20}) is unique and is given by (\ref{a16}) 
.\qed
\end{proof}

Given that $\int_{\mb{R}^n} G(x,\eta) \, d x^1 \dots d x^n = 1$ we obtain from (\ref{a16}) the bound, for $\xi \in M_\varepsilon$,
\be
|u - \bar{u}| \leq \eta \sup_{\xi \in M_\varepsilon} | \psi | \,.
\label{a30}
\ee

\section{Application}

We return to the example of an ideal fluid.
We set $u=(v,p)$, where $v = (v^1,\dots,v^n)$ is the filtered fluid velocity vector and $u^{n+1}=p$  is the filtered fluid pressure.
In this case $N=n+1$ and we set $F=(F_{(v)},F_{(p)})$, where $F_{(v)} = (F^1, \dots , F^n)$ and $F_{(p)}=F^{n+1}$.
Motivated by Section 1 we define the core function by
\be
F_{(v)} = v_t + v \cdot \n v + \n p \, , \quad F_{(p)} = \n \cdot v \, .
\label{ap1}
\ee
Note that the core function is not explicitly dependent on $\xi$ and $u_\eta$.

Since we insist that admissible solutions have well defined filter properties, i.e. $u \in \m{P}(M,\mb{R}^N)$, Proposition \ref{thme1} demands that the residual satisfy (\ref{me19}).
Here the source term $s$ is obtained by applying the identity (\ref{me20}) to the explicit representation of the core function given in (\ref{ap1}):
\be
s^a = - 2 \sum_{b=1}^n \sum_{c=1}^n v_c^b v_{bc}^a \, , \quad (1 \leq a \leq n) \, , \qquad \quad s^{n+1} =0  \, ,
\label{ap2}
\ee
using the notation $u_i=(v_i,p_i)$, $u_{ij}=(v_{ij},p_{ij})$.
We set $r=(r_{(v)},r_{(p)})$, where $r_{(v)} = (r^1, \cdots ,r^n)$ and $r_{(p)}=r^{n+1}$, and similarly $s=(s_{(v)},s_{(p)})$, where $s_{(v)} = (s^1, \cdots ,s^n)$ and $s_{(p)}=s^{n+1}$.
If we can assume that $r_{(p)} \rightarrow 0$ as $\eta \rightarrow 0^+$ then, in light of the residual equation (\ref{me19}) and the fact that the component of the source term $s_{(p)}=s^{n+1}$ vanishes identically on $M$, $r_{(p)}$ also vanishes identically on $M$.
It follows that the continuity equation $\n \cdot v =0$ is form invariant with repsect to scale.
We can now write
\be
s_{(v)} = - 2 \, \n \cdot \sigma \, , \qquad s_{(p)} =0 \, ,
\label{ap3}
\ee
where $\sigma$ is a symmetric second order tensor whose components are given by (\ref{sig}).
Thus (\ref{ap3}) is in agreement with the results presented in the first section.
We note that the source term in the residual equation also does not contain any explicit dependence on terms involving $\p_\eta$ and $\p_t$.
 
For $u \in \m{P}(M,\mb{R}^N)$ the macroscopic equations (\ref{me18}) then take the explicit form
\be
\f{\p v}{\p t} + v \cdot \n v + \n p = r_{(v)} \, ,
\quad \n \cdot v = 0 \, .
\label{ap4}
\ee
In practice we wish to solve (\ref{ap4}) (subject to suitable initial conditions) independently on scale slices $M|_{\eta=\text{const}}$.
Although the source term (\ref{ap3}) is suitably defined, the residual equation (\ref{me19}) is still not suitable for this purpose due to the presence of the $\p_\eta$ term on the left hand side.
For this reason we resort to an approximation of the residual based on (\ref{me19})-(\ref{me20}) with the aim of obtaining a member of $\m{R}(M,\mb{R}^N)$ that approximates some member of $\m{P}(M,\mb{R}^N)$.

In a similarly fashion as above, let us set $e = (e_{(v)},e_{(p)})$ and $\psi=(\psi_{(v)},\psi_{(p)})$.
Because $r_{(p)}$ and $s_{(p)}$ vanish indentically on $M$ it follows that $e_{(p)}$ also vanishes identically on $M$.
Using the prescription of the core function (\ref{ap1}) to obtain the coefficients, the system of PDEs (\ref{me30}) reduce to
\be
\p_t \psi_{(v)} + v \cdot \n \psi_{(v)} + \psi_{(v)} \cdot \n v + \n \psi_{(p)} = e_{(v)} \, , \qquad \n \cdot \psi_{(v)} = 0 \, .
\label{ap6}
\ee
for some $v \in \m{R}(M,\mb{R}^n)$ that is a solution of (\ref{ap4}).
Clearly the long time behaviour of (\ref{ap6}) is of interest and will depend crucially on the specific residual approximation used.
If one can supply reliable estimates of the residual equation error $e$ on the scale slice on which the space-time solution is being sought then there is the opportunity of resorting to computation.
Since the left hand sides of (\ref{ap6}) contain no partial derivatives with respect to scale we may, from the computational view, obtain numerical estimates of $\psi$ by coupling (\ref{ap6}) with the system (\ref{ap4}) using the same algorithms.
In this way one may obtain estimates of the error bound (\ref{a30}) on the scale slice of interest as the computations proceed.

Before concluding it is important to note that for any $u \in \m{P}(M,\mb{R}^N)$ satisfying the continuity constraint there will always exist an $r_{(v)} \in \m{P}(M,\mb{R}^N)$ satisfying (\ref{me19}) with source (\ref{ap3}) such that (\ref{ap4}) will hold.
This does not entirely resolve the existence problem for applications since many solutions will have very little interest in this respect.
As already mentioned, the problem of extracting from this set those solutions that are of interest to the application under study is a separate issue and must be examined in the context of the prescribed initial conditions and, to some extent, the solution method.

\section{Concluding Remarks}

In avoiding reference to a fully resolved solution we do not insist that regular fully resolved solutions do not exist, only that we do not insist that our space-time-scale solutions uniformly converge to one in the limit as $\eta \rightarrow 0^+$.
In doing so we have the expectation that meaningful solutions to applications may be obtained entirely within space-time-scale. 

We have already seen that for any member of $\m{P}(M,\mb{R}^N)$ there will exist an $r \in \m{R}(M,\mb{R}^N)$ satisfying (\ref{me18}) and (\ref{me19}) where the source term $s$ is given by (\ref{me20}).
But we have yet to formally establish that $\m{P}(M,\mb{R}^N)$ is sufficiently populated that it contains members that have any useful meaning for specific applications.
Allied to this is the construction of algorithms that allow us to extract from $\m{P}(M,\mb{R}^N)$ such solutions.

Some progress has been made in this direction through numerical solution methods.
For nonlinear systems the introduction of discretisation dictates that one is seeking an approximation of the filtered representation of the dependent variables associated with the scale of resolution rather than an approximation of the fully resolved dependent variables.
In this context there arises the associated problem of how one defines the core function.
We can no longer simply associate the core function with the PDEs that are to be assumed to govern the fully resolved system.
In Section 1 an attempt was made to rederive the core function of an ideal fluid application without reliance on the well known governing equations assumed for the fully resolved system.
At this stage, the problem of defining the core function is very much application specific and needs to be investigated further in a more general context.

\end{document}